\documentclass[10pt]{amsart}
\usepackage{amssymb}
\usepackage{amscd}

\newcommand{\bZ}{{\mathbb Z}}

\newcommand{\bC}{{\mathbb C}}

\newcommand{\bQ}{{\mathbb Q}}

\newcommand{\bT}{{\mathbb T}}

\newcommand{\bG}{{\mathbb G}}

\newtheorem{thm}{Theorem}[section]
\newtheorem{lemma}[thm]{Lemma}
\newtheorem{cor}[thm]{Corollary}

\numberwithin{equation}{section}

\begin{document}

\title[K-theories  of  Rost motives]{ Kunneth formula
for graded rings associated to $K$-theories of Rost motives}
 
\author{Nobuaki Yagita}

\address{ faculty of Education, 
Ibaraki University,
Mito, Ibaraki, Japan}
 
\email{nobuaki.yagita.math@vc.ibaraki.ac.jp, }


\maketitle

\begin{abstract}
In this paper, we study the graded ring $gr^*(X)$
defined by $K$-theory of a twist flag variety $X$.
In particular,  
the Kunneth map 
$gr^*(R')\otimes gr^*(R')\to gr^*(R)$
is studied explicitly for an original Rost motive $R'$ and
a generalized Rost motive $R$.
Using this, we give examples $T(X)^2\not =0$
for the ideal $T(X)$ of torsion elements in the Chow ring $CH^*(X)$.
\end{abstract}

\section{Introduction}
Let  $X$ be a smooth variety 
over a fixed field $k\subset \bC$.
We assume $X$ is geometric cellular, i.e.,
$\bar  X=X|_{\bar k}$ is cellular for the algebraic closure
$\bar k$ of $k$.
Let $p$ be a fixed prime and $CH^*(X)$ mean
the $p$-localized Chow ring
$CH^*(X)_{(p)}$ generated by algebraic cycles on $X$ modulo
rational equivalence.  In this paper, we study 
the ideal $T(X)$ in $CH^*(X)$ generated by torsion elements,
which is also the kernel
of the restriction map
\[ res_{CH}: CH^*(X)\to CH^*(\bar X)\cong H^*(X(\bC))_{(p)}\]
where $X(\bC)$ is the complex manifold of
$\bC$-rational points.

However known  cases $T(X)\not =0$ seem not so many. 
 Typical  examples are cases that $X$ contains the
Rost motive (\cite{Ro1},\cite{Vo2},\cite{Vo3}).
 (E.g.,  $X$ are some twisted flag varieties defined
from simple algebraic groups $G$ which have $p$-torsion in $H^*(G(\bC)) $  \cite{Pe-Se-Za}. Hence $p=2,3$ or $5$.)
  
Such cases have properties 
$T(X)^2=0$ (\cite{YaE}, \cite{YaC}).  We want to make examples with $T(X)^2\not =0$. 
If the Kunneth formula would hold, then $T(X\times X)^2\not =0$.
However, we know  $CH^*(X\times X)\cong CH^*(X)\otimes CH^*(\bar X)$, and this case $T(X\times X)^2=0$ also.

To get examples of torsion elements (e.g., $T(X)^2\not =0$)
we use 
the versal twisted flag variety (\cite{Ka1}, \cite{Me-Ne-Za}), in stead of $X\times X$,
and hence we need to consider arguments over some extension $k'$
of $k$.  In stead of $CH^*(X)$,
we use the integral algebraic Morava $\tilde K(m)$-theory \cite{YaA} (e.g., when $m=1$, it is essentially isomorphic to the usual algebraic $K^*$-theory).
 Let us write $gr(m)^*(X)$ the associated graded ring of the geometric filtration of the Morava $K^*$-theory.
Let $R_n$ be the (original) Rost motive for a nonzero pure symbol
in the $mod(p)$ $(n+1)$-dimensional Milnor $K$-theory $K_{n+1}^M(k)/p$. 

\begin{thm}  Let $p=2,3$ or $5$.  When $n=2$ and $m=1$, there is an
 irreducible motive $R$ (called generalized Rost motive) 
over a field $k'$  such that 
    \[ (gr(1)^*(R_2)\otimes gr(1)^*(R_2))/J_2\cong
gr(1)^*(R)\]
where $J_2$ is some ideal explicitly known
(see $\S 4$ below). 
\end{thm}
{\bf Remark.}   We can take the above $R, k'$ such that
$R$ is a summand of  a  versal twisted flag variety
over $k'$ (for details of versal flag varieties, see $\S 6$, below).
\begin{thm}
 When $p=2$ and $1\le m\le n-1$,
there is an irreducible motive $R$  over $k'$ such that there is an injection
 \[ (gr(m)^*(R_n)\otimes gr(m)^*(R_n))/J_2
\subset gr(m)^*(R)/J_2\]
where $J_2$ is the ideal defined in Lemma 7.2.
\end{thm}
Since there is an ideal $I(m)\subset CH^*(X)$ with
$CH^*(X)/I(m)\cong gr(m)^*(X)$, we see $T(X)^2\not =0$
for $X$ containing $R$ in the preceding theorems.
For each $s\ge 3$, the similar arguments work
for the $s$-th tensor product $gr(m)^*(R')^{\otimes s}$.
Thus we can show ;

\begin{cor}
Let $p=2,3$ or $5$.  For each $s\ge 1$, there is a twisted flag variety $X$ with $ T(X)^s\not =0$ for the torsion subgroup $T(X)$ of $CH^*(X)$.
\end{cor}
   
This paper is organized as follows.
In $\S 2$, we recall the geometric filtrations of
algebraic Morava $K$-theories.
In $\S 3$ we recall  Chow rings and Morava  $K$-theories
of the Rost motives $R_n$,
and in $\S 4$, we consider the Kunneth formula of  
$K$-theories for  $R_n$.
In $\S 5$ we consider when $X$ are flag varieties
containing Rost motives.
 In $\S 6$, we recall versal torsors and  study complete
flag varieties containing $R_2$.
In the last section, we study the maximal neighbor of the $(n+1)$-Pfister quadric which contains $R_{n-1}$.

\section{geometric filtration}

Let us fix a prime number $p$.
Recall (toplogical) $BP^*(X)$ is the cohomology theory with
the coefficient ring 
$BP^*=BP^*(pt.)\cong \bZ_{(p)}[v_1,v_2,...]$ for $|v_i|=-2(p^i-1)$.
The  complex cobordism theory $MU^*(X)_{(p)}$ localized at $p$
contains $BP^*(X)$ as a natural direct summand.
Moreover it is known that
\[MU^*(X)_{(p)}\cong BP^*(X)\otimes _{BP^*}MU_{(p)}^*\]

Let $h^*(-)$ be a (topogical) $BP$-oriented cohomology,
with  the coefficient ring $h^*\cong BP^*/I$ for some ideal
$I\subset BP^*$
(e.g., $H^*(-)_{(p)}, BP^*(-), k(i)^*(-)$).  (Moreover
$CH^*(X)$ means the $p$-localized theory  $CH^*(X)_{(p)}$ in this paper).  
Let $Ah^{*,*'}(X)$
be the corresponding motivic theory with
$Ah^{2*,*}(pt.)\cong h^*=h^*(pt.)$ \cite{YaA}.  Then we can consider the Atiyah-Hirzebruch spectral sequence  (\cite{YaA}, \cite{YaB}) 
\[E_2^{*,*',*''}\cong H^{*,*'}(X;h^{*''})\Longrightarrow
          Ah^{*,*'}(X).\]

It is known $H^{*,*'}(X)=0$ if $*>2*'$ and 
$H^{2*,*}(X)\cong CH^*(X)$.  
Hence $d_r|H^{2*,*}(X)=0$ and $E_{\infty}^{2*,*,0}\cong
CH^*(X)/I$ for some ideal $I$. 
Thus we have (\cite{YaA})                      
\begin{lemma}  There is a (geometric) 
filtration $(F^i)$ of $Ah^{2*,*}(X)$
and an ideal $I(h)\subset CH^*(X)$ such that
\[ CH^*(X)/I(h)\cong E_{\infty}^{2*,*,0}=\oplus_{i} F^i/F^{i+1}.\]
\end{lemma}

Let us write the above graded ring by
\[ gr(h)_{geo}^*(X)=\oplus F_i/F_{i+1},\]
and say the associated ring of the geometric filtaration
of $h^*$-theory.  

In this paper, we  mainly consider the algebraic Morava $K$-theory.  
For the coefficient ring $BP^*=\bZ_{(p)}[v_1,v_2,...]$, 
the integral connected Morava $K$-theory 
 $\tilde k(m)^*(-)$ is the cohomology theory 
with the coefficient ring 
\[ \tilde k(m)^*=BP^*/(v_1,...,\hat v_m,...)\cong \bZ_{(p)}[v_m],\] and 
$\tilde K(m)^*(X)=\tilde k(m)^*(X)[v_m^{-1}].$ 
For ease of notation, we write
\[ k_m^{2*}(X)=A\tilde k(m)^{2*,*}(X)\quad with\  \ k_m^*
=\bZ_{(p)}[v_m], \]
and $K_m^*(X)=k_m^*(X)[v_m^{-1}]$.
Here note that $K_1^*(X)$ is (essentially) isomorphic to
the usual algebraic $K^*$-theory.

{\bf  Remark.}  We note that
$ k_m^{2*}(X)=A\tilde k(m)^{2*,*}(X)\cong
ABP^{2*,*}(X)\otimes _{BP^*}k_m^*.$
However for $(2*,*')\not =(2*,*)$, in general, we know
\[A\tilde k(m)^{2*,*'}(X)\not \cong   
ABP^{2*,*'}(X)\otimes _{BP^*}k_m^*.\]

We have the isomorphism
\[ CH^*(X)\cong ABP^{2*,*}(X)\otimes _{BP^*}\bZ_{(p)}
 \cong k_m^{2*}(X)\otimes _{k_m^*}\bZ_{(p)}.\]
Hence $I(k_m)=0$.  We have $K_m^*(X)\cong k_m^*(X)[v_m^{-1}]$.  However  $I( K_m)$ is non-zero, in general,
which is generated by (higher) $v_m$-torsion elements
in $E_{\infty}^{2*,*,0}$  of the spectral sequence converging to $k_m^*(X)$.
Let us write simply
\[ gr(m)^*(X)=gr(K_m)_{geo}^*(X)\]
the graded ring of the geometric filtration of the 
integral algebraic Morava $A\tilde K(m)^{2*,*}$-theory,
and try to compute this graded ring in the next sections.

\quad

\section{ Morava K-theory of the Rost motive}

Let $\Omega^*(X)$ be the $BP^*$-version of the algebraic cobordism (defined by Levine-Morel in \cite{Le-Mo1}, \cite{Le-Mo2}) that is
\[ \Omega^{2*}(X)=ABP^{2*,*}(X)\cong MGL^{2*,*}(X)\otimes_{MU^*}BP^*,\]
where $MGL^{*,*'}(X)$ is the motivic cobordism theory
defined by Voevodsky \cite{Vo1}.

We recall the (original) Rost motive $R_a$
(we write it by $R_n$)
defined from a nonzero pure symbol $a$ in $K_{n+1}^M(k)/p$
(\cite{Ro1}, \cite{Ro2}, \cite{Vo1}, \cite{Vo2}, \cite{Vo3}).
Let $\bar R_n=R_n|_{\bar k}$ for the algebraic
closure $\bar k$ of $k$.
We have  isomorphisms (with degree $|y|=2(p^n-1)/(p-1)$)
\[ CH^*(\bar R_n)\cong \bZ[y]/(y^p),\quad \Omega^*(\bar
R_n)\cong BP^*[y]/(y^p). \]
(Here $|y|$ is the two time of the usual  degree for $y\in CH^*(X).$)
For a ring $A$, let us denote by $A\{a,...,b\}$ the free
$A$-module generated by $a,...,b$.  For a graded
ring $B$, let $B^+$ be the degree positive parts in $B$.
\begin{thm} (\cite{Vi-Ya}, \cite{YaB}, \cite{Me-Su})
The restriction 
$ res_{\Omega}:\Omega^*(R_n)\to \Omega^*(\bar R_n)$
is injective.  Let  $I_n=(p,...,v_{n-1})\subset BP^*$. Then
\[Im(res_{\Omega})\cong  BP^*\{1\}\oplus I_n\otimes \bZ_{(p)}[y]^{+}/(y^p)
\subset BP^*[y]/(y^p)\cong \Omega^*(\bar R_n).\]
Take  $c_i(y^j)\in \Omega^*(R_n)$ with 
 $res_{\Omega}(c_i(y^j))=v_iy^j$. Identifying $c_i(y^j)\in CH^*(R_n)$,   
\[CH^*(R_n)\cong \bZ_{(p)}\{1\}\oplus
\oplus_j\bZ_{(p)}\{c_0(y^j)\}\oplus \oplus_{i,j}\bZ/p\{c_i(y^j)\}\]
where $i,\ j$ range for  $1\le j\le p-1,\ 
1\le i\le n-1.$
\end{thm}
{\bf Example.}  In particular,   we have  isomorphisms
\[ CH^*(R_2)\cong \bZ_{(p)}\{1,c_0(y),...,c_0(y^{p-1})\}
\oplus \bZ/p\{c_1(y),...,c_1(y^{p-1})
\}.\]

\begin{lemma}  In $\Omega^*(R)$, we have the equation
\[ v_sc_r(y^j)-v_rc_s(y^j)=0\quad for\ 0\le r,s\le n-1.\]
\end{lemma}
\begin{proof}
In $\Omega^*(\bar R)$, we see
$ v_sc_r(y^j)-v_rc_s(y^j)=v_sv_ry^j-v_rv_sy^j=0.$
The injectivity of $res_{\Omega}$ implies the lemma.
\end{proof}

Now we compute the Morava $K$-theory of the Rost motive
$R_n$.   By definition,  
$k_m^*=BP^*/(v_1,...,\hat v_m,...),$
and hence $v_i=0\in k_m^*$ for $i\not =0,m$.
\begin{lemma} 
Let $R_a=R_n$ and $m\ge n$.  Then $k_m(R_n)\cong
k_m^*\otimes CH^*(R_n)$.
\end{lemma}
\begin{proof}
We see  $pc_i(y^j)=v_ic_0(y^j)=0\in k^*_m(R_n)$ for all $1\le i\le n-1$.   Since $k_m^*(R)\cong \Omega^*(R)\otimes_{BP^*}k_m^*$, we have  
\[k_m^*(R_n)\cong 
k_m^*\{1\}\oplus \oplus_i k_m^*\{c_0(y^j)\}\oplus \oplus_{i,j}k_m^*/p\{c_i(y^j)\},\]
which is isomorphic to $k_m^*\otimes CH^*(R)$.
\end{proof}
{\bf Remark.} In the Atiyah-Hirzebruch spectral sequence
\cite{YaA}
\[ H^{*,*'}(R_n;K_m^{*''})\cong E_2^{*,*',*''}\Longrightarrow A\tilde k(m)^{*,*'}(R_n),\]
the  nonzero differential is $d_{r+1}$ with
$r=0$ $mod(|v_m|)=mod(2p^m-2)$.
However, $|R_n|=2(p^n-1)$.  Hence if $m\ge n$,   then the spectral sequence collapses, in particular
\[ k_m^*(R_n)\cong k_m^*\otimes H^{2*,*}(R_n)\cong k_m^*\otimes CH^*(R_n).\] 
\begin{lemma}
Let $0\le m\le n-1$.  Then we have the $k_m^*$-module isomorphism $k_m^*(R_n)\cong k^*\{1\}\oplus KA\otimes \bZ[y]/(y^{p-1})$
 where  \[ KA=
k_m^*/(p,v_m)\{c_i(y)|i\not =0,m\}\oplus k_m^*\{c_0(y),c_m(y)\}/(pc_m(y)=v_mc_0(y)).\]
The restriction map induces the $K_m^*$-module isomorphism
\[K_m^{*}(R_a)\cong K_m^*\{1,c_m(y),....,c_m(y^{p-1})\}
\cong K_m^*[y]/(y^p)\cong  
K_m^*(\bar R_a).\]\end{lemma}
\begin{proof} 
From the preceding lemma, we see that
\[ v_mc_r(y^j)=v_rc_m(y^j)=0,\quad pc_r(y^j)=v_rc_0(y^j)=0\]
in $ k_m^*(R)$ for $ r\not =0,m.$
Since $k_m^*(R)\cong \Omega^*(R)\otimes _{BP^*}k_m^*$,
we have the first isomorphism.
Recall $K_m^*(X)\cong k_m^*(X)[v_m^{-1}]$.
In $K_m^*(R_n)$, we have $ c_0(y^j)=v_m^{-1}pc_m(y^j)$. Localized the first isomorphism by $v_m$, we get the second isomorphism.
\end{proof}

{\bf Remark.}  In Proposition 4.1 in \cite{Pe-Se1},
the $K_m^*$-motive decomposition of $R_n$ is given.
However we do note use $K_m^*$-motives in this paper.
We use $\Omega^*$-motives $R$ and $K_m^*(R)$
is defined just as $\Omega^*(R)\otimes_{BP^*}K_m^*$.

Let us write $gr_p(A)=A^0\oplus A^+/p\oplus pA^+$ for a graded ring $A$.
Recall the graded ring 
$gr(m)^*(X)=gr(K_m)_{geo}(X)$ defined in the preceding section.
\begin{cor}
There are isomorphisms
\[ gr(m)^*(R_n)\cong CH^*(R_n)/I(m)
\]
\[ \cong \bZ_{(p)}\{1,c_0(y),...,c_0(y^{p-1})\}\oplus 
\bZ/p\{c_m(y),...,c_m(y^{p-1})\}.\]where 
$I(m)=\bZ/p\{c_i(y^j)|i\not =0,m\}$.
We also have the (ungraded) $K_m^*$-module 
isomorphism
\[ K_m^*\otimes gr(m)(R_n)\cong gr_p(K_m^*(\bar R_n)).
\]
\end{cor}
\begin{proof}
Since $v_mc_i(y^j)=0\in k_m(R)$ for $i\not =0,m$, so $c_i(y^j)=0\in K_m^*(R)$.
Hence the first isomorphism is immediate.

 The second isomorphism follows from the isomorphism (for $CH^*(R)$) in Theorem 3.1 using $v_m^{-1}\in K_m^*$, in fact
\[ K^*_m\otimes CH^*(R)\cong
K_m^*/p\{c_m(y),...,c_m(y^{p-1}\}\oplus
K_m^*\{1,c_0(y),...,c_0(y^{p-1})\}\]
\[\cong K_m^*/p\{v_my,...,v_my^{p-1}\}\oplus
K_m^*\{1,py,...,py^{p-1}\}\]
\[ \cong K^*_m/p\{y,...,y^{p-1}\}\oplus
K_m^*\{1,py,...,py^{p-1}\}\cong gr_p(K_m^*[y]/(y^p)).\]
\end{proof}
\begin{cor}
When $n=2$, we see $I(1)=0$.
\end{cor}

The Rost motive $R_n=R_a$ is defined from the norm variety $V_n$ with $V_n(\bC)\cong v_n$ and $a|_{k(V_n)}=0$.
In fact, there is $y_2\in CH^*(V_n\times V_n)$ \cite{Ro2}
with $|y_2|=2(p^n-1)/(p-1)$ such that 
$R_n=pr_{\theta}(V_n)$ for the  projector $pr_{\theta}$
defined from the correspondence $\theta=y_2^{p-1}
\in CH^{dim(V_n)}(V_n\times V_n)$.
In particular 
\[ CH^*(R_n\times R_n)\cong CH^*(R_n)\otimes
\bZ_{(p)}[y_2]/(y_2^p)\cong CH^*(R_n)\otimes CH^*(\bar R_n).\] 
\begin{lemma} 
For $h=CH,\Omega,k_m,K_m$ or $h=gr(m) $
with $h^*=\bZ_{(p)}$,  we have
\[  h^*(R_n\times R_n)\cong h^*(R_n)\otimes_{h^*} h^*(\bar R_n).\]
\end{lemma}

\section{Kunneth maps for Rost motives}

Let us write
\[ R'=R_{n_1},\ \ R''=R_{n_2}\quad with\ 
1\le m\le min(n_1,n_2).\]
Let us write
$CH^*(\bar R')\cong \bZ[y_1]/(y_1^p)$
and  $CH^*(\bar R'')\cong \bZ[y_2]/(y_2^p)$, and 
\[ C(R',R'')=gr(m)^+(R')\otimes gr(m)^+(R'')\]
\[
\cong CH^+(R')/I'(m)\otimes CH^+(R'')/I''(m).\]
Then we have the decomposition $C(R',R'')
= C_0\oplus C_1\oplus C_2$
\[where\quad C_s=\begin{cases}
 \oplus_{1\le i,j\le p-1}\bZ\{c_0(y_1^i)c_0(y_2^j)\}
\qquad if\ s=0\\
   \oplus_{1\le i,j\le p-1}    \bZ/p\{c_m(y_1^i)c_m(y_2^j)\}\qquad if\ s=1\\
 \oplus_{1\le i,j\le p-1}\bZ/p\{c_m(y_1^i)c_0(y_2^j), c_0(y_1^i)c_m(y_2^j)\}\quad    if\ s=2.
 \end{cases}  \]

We consider some  motive $R$ such that there are maps
$R\to R',R"$, and hence $CH^*(R')\otimes CH^*(R'')\to CH^*(R)$ which induces over $\bar k$
\[(*)\quad CH^*(\bar R')\otimes CH^*(\bar R'')\stackrel{\cong}{\to}CH^*(\bar R).\]
Of course $R=R'\times R''$ satisfies the above conditions,
however, we mainly consider other cases.

Let $J_2\subset C_2$  be defined by \[J_2=\oplus_{1\le i,j\le p-1}\bZ/p\{c_m(y_1^i)c_0(y_2^j)- c_0(y_1^i)c_m(y_2^j)\}.\]
Let $gr_{p^2}(A)=A^0\oplus A^+/p\oplus pA^+/p^2A^+\oplus p^2A^+$.  Let us write
\[ \tilde gr(m)(R)= gr(m)(R)/(gr(m)(R')+gr(m)(R'')), \quad \]
\[\tilde K^*_m(\bar R)=K^*_m(R)/(K^*_m(\bar R')+K^*_m(\bar R'')).\]

\begin{lemma}  Let $R$ satisfies $(*)$.
We have an isomorphism
\[ K^*_m\otimes C(R',R'')/(J_2)
\cong 
gr_{p^2}(\tilde K^*_m(\bar R)).\]
The map $ j: C(R',R'')/J_2 \to \tilde gr(m)(R)/J_2$
is injective, if we have, in the connective $k_m^*(\bar R)$, 
\[ (**)\qquad  py_1^iy_2^j,\ v_my_1^iy_2^j\not \in Im(res_{k_m})\quad
for \ 1\le i,j\le p-1.\]
 \end{lemma}
\begin{proof} 
We first note that 
\[res_{K}:K_m^*(R)\to K_m^*(\bar R)\cong K_m[y_1,y_2]/(y_1^p,y_2^p)\]
is surjective, since $c_m(y^j_i)\mapsto v_my_i^j$.
We also note that
\[res_{K}(c_m(y_1^i)c_0(y_2^j)- c_0(y_1^i)c_m(y_2^j))
=v_my_1^ipy_2^j-py_1^iv_my_2^j=0\in K^*_m(\bar R).\]
Hence we see $res(J_2)=0$.

Let $C_2'=\oplus_{1\le i,j\le p-1}\bZ/p\{c_m(y_1^i)c_0(y_2^j)\},$ so that
$C_2'\cong C_2/J_2$.
We can show the isomorphism in this lemma.
In fact, we see \ \  
$K^*_m\otimes C_0\cong p^2\tilde K^*_m(\bar R),$ \
$K^*\otimes C_1\cong \tilde K^*(\bar R)/p$
from $v_m^{-1}\in K_m^*$, \ \ and
\[K_m^*\otimes C_2/J_2\cong  K^*_m\otimes C_2'\cong p\tilde K^*_m(\bar R)/p^2.\]

Next, we consider the map 
\[ k_m^*(R')\otimes_{k_m^*}k_m^*(R'')\to 
k_m^*(R)\to k_m^*(\bar R)\cong k_m^*[y_1,y_2]/(y_1^p,y_2^p).\]
Take $x=c_m(y_1^i)c_0(y_2^j)\in C_2'$ so that
$res_{k_m}(x)=pv_my_1^iy_2^j$ identifying
$x\in k_m^*(R)$.
Moreover, suppose $j/p(x)=0$  for $j/p : C(R',R'')/(p,J_2)\to gr(m)^*(R)/(p,J_2)$. 
  Then
\[ pv_my_1^iy_2^j\in (p,v_m)Im(res_{k_m})\subset k_m^*(\bar R).\]
This means that 
$py_1^iy_2^j\in Im(res_{k_m})$
or $v_my_1^iy_2^j\in Im(res_{k_m}).$
This contradicts to $(**)$ in this lemma.

We also show that $(**)$ implies $j/p(x)\not=0$ in the
cases $x\in C_1$ and $x\in C_0$.
So $(**)$ implies the injectivity of $j/p$.
 Recall that $C_1,C_2$ are just $p$-torsion and 
$j|C_0$ is injective.  Hence the injectivity of $j/p$ implies
so of $j$. 
\end{proof}

{\bf Remark.}
When $R'=R''$ and $R=R'\times R''$ satisfy $(*)$.
From Lemma 3.7,  we know over $k$,
\[ gr(m)^*(R)\cong gr(m)^*(R')\otimes gr(m)^*(\bar R')
.\]
Hence $j$ is not injective, e.g., $j(c_0(y_1)c_m(y_2))=v_mc_0(y_1)y_2=0,$
since $y_2\in gr(m)^*(\bar R').$
Note  $py_1y_2, v_my_1y_2\in Im(res_{k_m})$.  

\begin{cor} If the map $j: C(R',R'')/J_2\to \tilde gr(m)^*(R)/J_2$
is surjective, then it is isomorphic.
\end{cor}
\begin{proof}
From the surjectivity of $j$, the image of $res_{k_m}$ is generated by  the image  of
 $c_k(y_1^i)c_{\ell}(y_2^j)$
for $k,\ell=0$ or $m$. These are
\[p^2y_1^jy_2^j, \ pv_my_1^iy_2^j\ \ or\ \ v_m^2y_1^iy_2^j.\]
Hence $(**)$ in the preceding lemma is satisfied.
\end{proof}

\section{flag varieties}

Let $G$ be a (compact) Lie group,
and  $T\subset P\subset G$ be a  maximal torus and a parabolic subgroups.
 Let us write by  
$G_k$ and $T_k,P_k$ the  split  reductive group and split maximal torus ( and the parabolic group) over a field $k$ with $ch(k)=0$, 
corresponding to  Lie groups  $G$ and  $T,P$.  Let $B_k$ be the Borel
subgroup containing $T_k$.
Moreover let
$\bG$ be  a non-trivial  $G_k$-torsor.
 Then  
$X=\bG/P_k$ is a twisted flag variety.

Petrov, Semenov and Zainoulline develop the theory of generically splitting varieties.
We say that  $L$ is splitting field of a variety of $X$ if $M(X|_L)$ is isomorphic to a
direct sum of twisted Tate motives $\bT^{\otimes i}$
and the restriction map $i_L: M(X)\to M(X|_L)$ is isomorphic after tensoring $\bQ$.
A smooth scheme $X$ is said to be generically split over $k$ if its function field $L=k(X)$
is a  splitting field.  Note that  the complete flag variety $X=\bG/B_k$
is always generically split.

\begin{thm} (Theorem 3.7 in \cite{Pe-Se-Za})
Let $Q_k\subset P_k$ be parabolic subgroups of $G_k$ which are generically split over $k$.
Then there is a decomposition of motive 
\[M(\bG/Q_k)\cong M(\bG/P_k)\otimes H^*(P/Q)\]
where we identify $H^*(P/Q)\cong \oplus_i \bT^{i\otimes}$,
the sum of (degree changing) Tate motives.
\end{thm}

\begin{thm} (Theorem 5.13 in \cite{Pe-Se-Za}, 
\cite{Se})
The ($p$-localized) motive  $M(X)$ of $X=\bG/B_k$ is decomposed as
\[ M(X)_{(p)}=M(\bG/B_k)_{(p)}\cong 
R(\bG)\otimes(\oplus_i \bT^{\otimes s_i})\]
where $\bT$ is the Tate motive and $R(\bG)$ is  some
irreducible motive (called generalized Rost motive).
\end{thm} 

It is known that 
the  
 generalized Rost motive $R(\bG)$ is 
isomorphic to  the original 
Rost  motive $R_n$ in the following cases.

(i)\quad  cases $G$ are of type (I) \\ 
We  consider here the cases
\[ gr H^*(G;\bZ/p)\cong \bZ/p[y]/(y^p)\otimes 
\Lambda(x_1,...,x_{\ell})\]
where the degree are $|y|=2(p+1)$, $|x_i|=odd$,
and where $\Lambda(x_1,...,x_{\ell})$ is the exterior algebra
generated by $x_1,...,x_{\ell}$.
The following simple Lie groups satisfy 
the above isomorphism \cite{Mi-Tod};
\[ (G,p)=\begin{cases}   G_2,\ F_4,\ E_6, \ Spin(7),\ Spin(8),\ Spin(9)\quad for \ p=2\\ F_4,\ E_6,\ E_7\quad for \ p=3\\
E_8\quad p=5.
\end{cases}\]
We call that these (simply connected) groups $G$ are of type $(I)$.  Then from Petrov-Semenov-Zainoulline,
we know $R(\bG)\cong R_2$.  From Lemma 2.1 and Corollary 3.6,
we get
\begin{lemma}  We have the isomorphism
$gr(1)^*(X)\cong CH^*(X)$.
\end{lemma}

Let $BB_k$ be the classifying space of $B_k$.
Then we have the classifying map
\[c:CH^*(BB_k)\to CH^*(X)=CH^*(\bG/B_k).\]
Let us write
$CH^*(BB_k)=S(t)\cong \bZ_{(p)}[t_1,...,t_{\ell}]$
where $|t_i|=2$ and $\ell=rank_pG$.  
There is a regular sequence  $(b_1,...,b_{\ell})$ in $S(t)$ such that  we have the graded additive isomorphism 
\[ CH^*(\oplus_i\bT^{\otimes s_i})
\cong S(t)/(b_1,...,b_{\ell}).\]
The Chow ring
of $X$ is known.
\begin{thm} (\cite{Ro2},\cite{YaC})
Let $G$ be  of type $(I)$ and $rank(G)=\ell$.
Then $2p-2 \le \ell$, and we can take $b_i\in CH^*(BB_k)$ for $1\le i\le \ell$
such that $ c(b_{2i-1})=c_1(y^i)$ and $c(b_{2i})=c_0(y^i)$, 
and there are isomorphisms
\[ CH^*(\bar X)/p\cong \bZ/p[y]/(y^p)\otimes S(t)/
(b_{i'}|0\le i'\le \ell)\]
\[ CH^*(X)/p\cong S(t)/(p, b_ib_j,b_k|1\le  i,j\le 2p-2<k\le \ell).\]
\end{thm}

(ii) \quad The Pfister quadrics.

Let  $(G,p)=(SO(2^{n+1})+2,2)$ and $P=SO(2^{n+1})\times SO(2)$.  Then the quadric $X_q$ defined by the  
$(n+1)$ Pfister form $q$ is written as
$ X_q\cong G/P$.
It is well known that $R(\bG)\cong R_n$.
For ease of arguments, we consider the maximal
neighbor of the $(n+1)$ Pfister form $q'$. Then
\[X_{q'}\cong SO(2^{n+1}+1)/(SO(2^{n+1}-1)\times SO(2))\]
is generically split and 
 $R(\bG)\cong R_n$ also.

\begin{thm} (\cite{Ro1},\cite{Tod-Wa},\cite{YaE})
Let $X=X_{q'}=\bG/P_k$ be the maximal neighbor of the $(n+1)$ Pfister quadric.  Then  we have the ring isomorphisms 
\[ CH^*(\bar X)\cong \bZ_{(2)}[y,h]/(y^2,u_0=2y)\quad
where\ \  u_0=h^{2^{n}-1},\  |h|=2,\]
\[CH^*(X)\cong 
 \bZ_{(2)}[h,u_1,...,u_{n-1}]/(u_iu_j,2u_k)\]
where $0\le i,j,k\le n-1, k\not =0$ and 
 $u_i=c_i(y)=v_iy\in \Omega^*(\bar X)$.
\end{thm}

{\bf Remark.}  In \cite{YaE}, we have the $\bZ_{(2)}[h]$-module isomorphism
\[ CH^*(X)\cong \bZ_{(2)}[h]/(h^{2^{n+1}-2})\oplus \bZ/2[h]/(h^{2^n-1})\{u_1,...,u_{n-1}\}\]
with the multiplication $u_iu_j=0$.  Identifying $h^{2^n-1}=u_0$, 
the isomorphism in the theorem is immediate.

\begin{lemma}  For $X=X_{q'}$, we have the isomorphism
\[gr(m)^*(X)\cong \begin{cases}
  CH^*(X)/I(m)\quad if\  1\le m\le n-1\\
  CH^*(X)\quad if \ m\ge n.
\end{cases} \]
where $I(m)=Ideal(u_1,...,\hat u_m,...,u_{n-1}). $
\end{lemma}

More generally, a quadric $X_{\xi}$ is called  excellent if
its motive $M(X_{\xi})$ is isomorphic to  a  direct sum of  Rost motives
$R_s\otimes \bT^{r}$.  
\begin{thm} (\cite{YaE})
Let $X_{\xi}$ be an odd dimensional excellent anisotropic quadric with
$2^n-1\le dim(X_{\xi})=d\le 2^{n+1}-2$.
Then there are elements $c_1(d),...,c_{n-1}(d)$ 
in $CH^*(X_{\xi})$ and positive integers
$d_1(d)\ge...\ge d_{n-1}(d)$ such that
\[CH^*(X_{\xi})\cong \bZ_{(2)}[h,c_1(d),...,c_{n-1}(d)]/R,\]
\[\quad where\ \ R=(h^{d+1},h^{d_i(d)}c_i(d),2c_i(d),c_i(d)c_j(d)|
1\le i,j\le n-1).\]
 \end{thm}

{\bf Remark.}
From $CH^*(X)$ in the preceding three theorems, we
see $T(X)^2=Ker(res_{CH})^2=0$ as stated  in the introduction. In fact, $b_ib_j=0$ in $T(X)=(b_1,b_3,...,b_{2p-3})$ for  Theorem 5.4,  $u_iu_j=0$ in $T(X)=(u_1,...,u_{n-1})$ for  Theorem 5.5,
and $c_i(d)c_j(d)=0$  in $T(X)=(c_1(d),...,c_{n-1}(d))$ in Theorem 5.7.

\section{$m=1$ and versal flag varieties}

In this section, we consider the case $m=1$ related to
algebraic groups.  Of course $K^*_1(X)$ essentially is
isomorphic to  the algebraic 
$K$-theory $K^*(X)$ of $X$.   We simply write $K_1^*(X)$
by $K^*(X)$.  It is well known from Panin (\cite{Pa})
that $K^*(\bG/P_k)$  is torsion free for each parabolic subgroup.  
\begin{lemma} The restriction map
$res_K:K^*(\bG/P_k)\to K^*(\bar \bG/P_k)$ is injective.
\end{lemma}

Moreover if $G$ is simply connected, then $res_K$
is isomorphic for $P_k=B_k$ from a Chevalley result.

\begin{cor} Let $G$ be simply connected. Then we have
\[ K^*(G/T)\cong K^*(\bar \bG/B_k)\cong 
K^*(\bG/B_k).\]
\end{cor}

Let $gr_{\gamma}(X)$ be the associated graded ring
for the gamma filtration of $K$-theory  (see \cite{Ga-Za},\cite{Za}, \cite{YaF},
\cite{YaG} for details of gamma filtrations).
Note that gamma filtrations are defined both on
algebraic and topological $K$-theories.  From the above lemma
\begin{cor}  We  have $gr_{\gamma}(G/T)\cong gr_{\gamma}(\bG/B_k)$.
\end{cor}

Let us consider an embedding of $G_k$ into the general linear group $GL_N$ for some $N$.  This makes $GL_N$ a $G_k$-torsor over the quotient variety $S=GL_N/G_k$.
Let $F$ be the function field $k(S)$ and  define
the $versal$ $G_k$-$torsor$ $E$ to be the $G_k$-torsor over $F$ given by the generic fiber of $GL_N\to S$. 
(For details, see \cite{Ga-Me-Se}, \cite{To1}, \cite{Me-Ne-Za}, \cite{Ka1}.) 
\[\begin{CD}
        E@>>> GL_N\\
       @VVV     @VVV \\
     Spec(k(S))  @>>> S=GL_N/G_k
\end{CD}\]
The corresponding flag variety $E/B_k$ is called a $versal$ flag
variety, which is considered as the most complicated twisted
flag variety.  It is known that the Chow ring
$CH^*(E/B_k)$ is not dependent to the choice
of  generic $G_k$-torsors $E$ (Remark 2.3 in \cite{Ka1}).

Karpenko and Merkurjev  showed the following result
 for a versal  flag variety. 
\begin{thm}
(Karpenko Lemma 2.1 in \cite{Ka1})
Let $h^*(X)$ be an oriented cohomology theory
(e.g., $CH^*(X)$, $\Omega^*(X)$).
Let  $\bG/B_k$ be a versal flag variety.
Then the natural map
$h^*(BB_k)\to h^*(\bG/B_k)$ is surjective.
\end{thm}
\begin{cor}
If $\bG$ is versal, then $CH^*(\bG/B_k)$
is multiplicatively generated by elements $t_i$ in $S(t)$. 
\end{cor}
It is well known \cite{At}, \cite{YaF} that
if $gr_{geo}(X)$ is multiplicatively generated by Chern classes in
$CH^*(X)$, then $gr_{\gamma}(X)\cong
gr_{geo}(X)$.
\begin{cor}  If $\bG$
be versal, then $gr_{\gamma}(\bG/B_k)\cong gr_{geo}(\bG/B_k)$.
\end{cor}

Now we consider the case $G=G_1\times G_2$
with $G_i$ are of type $(I)$.  Let 
us write by $R^i$ for $R(\bG_i)$. (So $R^i\cong R_2$.)  
Let us write
$CH^*(\bar R^i)\cong \bZ_{(p)}[y_i]/(y_i^p)$.
Take $R=R(\bG)$  for a  versal torsor $\bG$
over $k(S)$ with $S=GL_N/(G_1\times G_2)$.  Moreover let  $R^i=R(\bG_i)|_{k(S)}$ for versal $\bG_i$, while they are defined over $k(S')$ with $S'=GL_N/(G_i)$.
Then these $R,R',R''$ satisfies $(*)$ in Lemma 4.1.

Recall $J_2=0$ in $K^*(\bar R)$.  From Corollary 6.2, we see $K^*(R)\cong K^*(\bar R)$.
Hence $J_2=0$ in $K^*(R)$ and so in $ gr_{\gamma}(R)$. 
From Corollary 4.2 and the above Karpenko result,
the following lemma is immediate
\begin{lemma}
Let $G=G_1\times G_2$ with $G_i$ of type $(I)$,
and  $\bG$ be  versal.  Then the map
$j: CH^*(R^1)\otimes CH^*(R^2)\to CH^*(R)$
is surjective, and 
we have isomorphisms
\[  gr_{\gamma}(R)\cong gr(1)^*(R)\cong (CH^*(R^1)\otimes CH^*(R^2))/(J_2),\] 
\[ K^*\otimes (\tilde gr_{\gamma}(R)) 
\cong 
gr_{p^2}(\tilde K^*(\bar R)).\]
 \end{lemma}

{\bf Remark.} In particular, let $p=2$.  Then
\[ CH^*(R)\cong\ \  CH^*(R_1)\otimes CH^*(R_2)\ \ 
 or\ \ 
CH^*(R_1)\otimes CH^*(R_2)/J_2.\]
Karpenko conjectures $CH^*(X)\cong gr_{\gamma}(X)$
in such cases.
So if it is correct, then we have the second case.

Next consider the case  $G=\times _{s=1}^nG_s$ with $G_s$ is of type $(I)$,  with $CH^*(\bar R_s)\cong \bZ_{(p)}[y_s]/(y_s^p)$.
Let $J_n\subset \otimes_{s=1}^nCH^*(R^s)$  be defined by \[J_n=\bZ/p\{c_1(y_r^i)c_0(y_t^j)- c_0(y_r^i)c_1(y_t^j)|1\le i,j\le p-1,\ 1\le r<t\le n \}.\]
Let us write $gr_{p^n}(A)=A^0\oplus A^+/p\oplus pA^+/p^2A^+\oplus...
\oplus p^nA^+$.

\begin{thm}
Let $G=\times _{i=1}^sG_i$ with $G_s$ of type $(I)$, and $\bG$ be versal.  Let $R,R^i$ are the Rost motives for $G,G_i$
respectively. 
Then we have isomorphisms
\[ gr_{\gamma}(R)\cong (\otimes_{i=1}^sCH^*(R^i))/(J_s),\]
\[  \tilde gr_{\gamma}(R)\cong K^*\otimes(\otimes_{i=1}^sCH^+(R^i)/(J_s))
\cong gr_{p^s}(\tilde K^*(\bar R)).\]
Here  $\tilde gr_{\gamma}=gr_{\gamma}(R)/\sum_{i=1}^s
gr_{\gamma}(R^{-i})$ and  $\tilde K^*(R)=K^*(R)/(\sum_{i=1}^s 
K^*(R^{-i}))$  where $R^{-i}$ is the (Rost) motive for
the group $G_{-i}=G_1\times ...\times\hat G_i\times...\times G_s$.
\end{thm}

\begin{cor}  Let $G=\times_{i=1}^sG_i$ and $G_i$ be of type $(I)$.  Then we have the isomorphisms
\[ gr_{\gamma}(G/T)\cong gr_{geo}(\bG/B_k)
\cong (\otimes _{i=1}^s gr_{\gamma}(G_i/T_i))/(J_s).\]
\end{cor}

\section{Pfister quadrics}

In this section, we consider the case 
  $(G,p)=(SO(2^{n+1}+1),2)$ and $P=SO(2^{n+1}-1)\times SO(2)$.  Then the quadric $X_{q'}$ defined by the  
maximal neighbor $q'$ of the 
$(n+1)$ Pfister form is written as
$ X_{q'}\cong G/P$.
Throughout this section, we assume $1\le m\le n-1$.
Recall Theorem 5.5 and Lemma 5.6.
\begin{lemma}  We have, for $u_0=h^{2^n-1}$, 
\[gr(m)^*(X_{q'})\cong
CH^*(X_{q'})/I(m)\cong 
\bZ_{(2)}[h,u_m]/(u_0^2, u_0u_m,u_m^2,2u_m^2).\]
\end{lemma}
{\bf Remark.}
If $u_1$ is represented by a Chern class, then we have
$gr_{\gamma}(X_{q'}(\bC))\cong gr(1)^*(X_{q'})$ as above.

Next we consider the product of Pfister forms.
\begin{lemma}
Let $\ell=2^n-1$.  Let $G=G_1\times G_2$ with $G_i\cong SO(2\ell +1)$. 
$\bG$ be  versal and $R=R(\bG)$.  Then the following map
\[j: (gr(m)^*(R^1)\otimes gr(m)^*(R^2))/J_2\to gr(m)^*(R)/J_2\]
is injective  for $J_2=(c_0(y_1)c_m(y_2)-c_m(y_1)c_0(y_2))$.
 \end{lemma}
\begin{proof}
Consider the diagram for $X'=\bG/P_k$, $X_i'=\bG_i/P_k$,
$X=\bG/B_k$ and $X_i=\bG_i/B_k$ for $i=1,2$.
\[\begin{CD}
 CH^*(X_1')\otimes CH^*(X_2')@>>>CH^*(X')@>{res'}>> CH^*(\bar X')@>>>\bZ_{(2)}[y,y']/(y^2,(y')^2)\\
 @VVV @VVV  @VVV @V{=}VV  \\
 CH^*(X_1)\otimes CH^*(X_2)@>>>
CH^*(X)@>{res'}>> CH^*(\bar X)@>>>\bZ_{(2)}[y,y']/(y^2,(y')^2)
\end{CD}.\]
Here  $y=y_{2\ell}$ and $y'=y_{2\ell}'$ in 
$CH^*(R)$ identifying
\[ CH^*(\bar R^1)/2 \cong \bZ/2[y_2,...,y_{2\ell})]/(y_2^2,...,y_{2\ell}^2)\quad |y_{2i}|=2i.\]

Recall that the torsion index $t(G_i)=2^{\ell}$.
(Note that the fundamental class  is written as $f=y_{top}t_{top}$
for $t_{top}\in S(t)$ and  $y_{top}=y_{2}y_{4}...y_{2\ell}$, 
 for details see \cite{To1}, Theorem 7.3 in \cite{YaC}).
Hence $t(G)=t(G_1)t(G_2)=2^{2\ell}$.  If 
$2y(y')\in Im(res_{CH})$,  then
$ t(G)\le 2^{2\ell-1}$ 
because  $2y_{2j}$, $2y_{2j'}'$ are in $Im(res_{CH})$
for all $j,j'$.
This is a contradiction and we see $2y(y')\not \in Im(res_{CH})$.

  Recall the  Quillen operation $r_{\Delta_m}$
so that $r_{\Delta_m}(v_m)=2$ in $\Omega^*(R)$
(see \cite{Ha}, \cite{YaA}, \cite{YaC} for the properties of $r_{\Delta_m}$). If $v_myy'\in Im(res_{\Omega})$, then
the fact $r_{\Delta_m}(v_myy')=2yy'$ imlplies $2yy'\in Im(res_{\Omega})$, and this is a contradiction.
Hence 
we also  see $v_myy'\not \in Im(res_{\Omega})$.
Therefore the map $j$ is injective from  Lemma 4.1.
\end{proof}
{\bf Remark.} If $K_m^*(R)$ is torsion free
(e.g., for $m=1$), then $J_2=0$ in $gr(m)^*(R)$.

\begin{cor}  Let $G=\times_{i=1}^sG_i$ with 
 $G_i=SO(2^{n+1}-1)$ and $\bG$ be versal.  Then we have the injection
\[  (\otimes _{i=1}^s gr(m)^*(\bG_i/P_i))/(J_s)
\subset  gr(m)^*(\bG/P_k)/J_s.\]
\end{cor}

\begin{proof}[Proof of Corollary 1.3 in the introduction.]
Take $X=\bG/P_k$.  Let us write by $c_m(y_i)$ the element $u_m\in gr(m)^*(G_i/P_i)$
so that $c_m(y_i)=v_my_i$ in $k_m^*(\bar G_i/P_i)$. 
Then the element
\[x=c_m(y_1)c_m(y_2)...c_m(y_s)\not =0\quad in \ (\otimes_{i=1}^sgr(m)^*(G_i/P_i)))/J_s,\]
which is also nonzro in $gr(m)^*(X)/J_s\cong CH^*(X)/(I(m),J_s)$.  So $x\in T(X)^s$ and nonzero
in $CH^*(X)$.
For $G$ of type $(I)$ , the case $X=\bG/B_k$ is proved
similarly letting $m=1$. 
\end{proof}
{\bf Remark.} For each example $X$, if $x\in T(X)$, then it seems $px=0$ and $x^p=0$.

{\bf Remark.}  For each prime $p$ and $n\ge 2$,
we have the norm  variety $X=V_n$
which contains $R_n$.  Hence $T(X)\not=0$.

{\bf Question.}
Is there a geometric cellular algebraic variety  $X$ such that the following
(1) or (2)  is satisfied ? 

(1)  \quad $T(X)^s\not =0$ for some $s\ge 2$ and $p\ge 7$,

(2)\quad there is an element  $x\in T(X)$ with $x^{p}\not =0$.

\quad

\end{document}